\documentclass[12pt]{amsart}
\usepackage{amscd, amssymb,latexsym,amsmath, amscd, amsmath}
\usepackage{amsthm}
\usepackage{amstext}
\usepackage{anysize}
\usepackage[sc]{mathpazo} 
\usepackage{enumitem}
\usepackage{pst-node}
\marginsize{2.5cm}{2.5cm}{2.5cm}{2.5cm}
\usepackage{graphics}
\usepackage{color}
\usepackage{tikz}
\usepackage{tikz-cd}

\DeclareFontFamily{U}{wncy}{}
    \DeclareFontShape{U}{wncy}{m}{n}{<->wncyr10}{}
    \DeclareSymbolFont{mcy}{U}{wncy}{m}{n}
    \DeclareMathSymbol{\Sh}{\mathord}{mcy}{"58}

\begin{document}
\title{Reducible principal series representations, and Langlands parameters for real groups}
\author{Dipendra Prasad}
\maketitle
{\hfill \today}

\vspace{1cm}

\theoremstyle{plain}

\newtheorem{theorem}{Theorem}
\newtheorem{proposition}{Proposition}
\newtheorem{conjecture}{Conjecture}
\newtheorem*{thm}{Theorem}
\newtheorem{cor}{Corollary}
\newtheorem{lemma}{Lemma}



\theoremstyle{definition}
\newtheorem{remark}{Remark}
\newtheorem{definition}{Definition}
\newtheorem{question}{Question}
\newtheorem{example}{Example}

\def\wG{{\widehat{G}}}
\def\wT{{\widehat{T}}}
\def\S{{\mathbb S}}
\def\A{{\mathbb A}}
\def\C{{\mathbb C}}                               
\def\F{{\mathbb F}}                               
\def\Ga{{\mathbb G}_a}                            
\def\Gm{{\mathbb G}_m}                            
\def\H{{\mathbb H}}                               
\def\K{{\mathbb K}}
\def\L{\mathcal L}
\def\M{{\mathbb M}}                               
\def\N{{\mathbb N}}                               
\def\O{{\mathcal O}}
\def\Q{{\mathbb Q}}                               
\def\R{{\mathbb R}}                               
\def\Z{{\mathbb Z}}                               
\def\PP{{\mathbb P}}                               

\def\k{{\overline{k}}}                            
\def\m{\mathfrak{m}}
\def\p{\mathfrak{p}}
\def\Aut{\operatorname{Aut}}                      
\def\End{{\operatorname{End}}}                    
\def\GL{\operatorname{GL}}                        
\def\GSO{\operatorname{GSO}}                        
\def\GSp{\operatorname{GSp}}                        
\def\Gal{\operatorname{Gal}}                      
\def\Hom{\operatorname{Hom}}                      
\def\Ind{\operatorname{Ind}}
\def\ind{\operatorname{ind}}

\def\P{{\mathbb{P}}}
\def\PSL{\operatorname{PSL}}                      
\def\PSO{\operatorname{PSO}}                      
\def\PSU{\operatorname{PSU}}                      
\def\PGL{\operatorname{PGL}}                      
\def\PSp{\operatorname{PSp}}                      
\def\Rad{\mathfrak{Rad}}
\def\Rep{\mathfrak{Rep}}
\def\Res{\operatorname{Res}}
\def\SL{\operatorname{SL}}           
\def\alg{\operatorname{alg}}           
\def\SL{\operatorname{SL}}                        
\def\SO{\operatorname{SO}}                        
\def\SU{\operatorname{SU}}                        
\def\Sp{\operatorname{Sp}}                        
\def\Spin{\operatorname{Spin}}                    
\def\Sym{\operatorname{Sym}}
\def\Sel{\operatorname{Sel}}
\def\Or{\operatorname{O}}                         
\def\Un{\operatorname{U}}                         
\def\Ps{\operatorname{Ps}}                         

\def\Fr{\operatorname{Fr}}                         
\def\det{\operatorname{det}}                      
\def\tr{\operatorname{tr}}                        
\def\into{{~\hookrightarrow~}}                    
\def\iso{{\stackrel{\sim}{~\longrightarrow~}}}    
\def\lra{{~\longrightarrow~}}                     %
\def\onto{{~\twoheadrightarrow~}}                 
\def\op{{\oplus}}                                 
\def\ot{{\otimes}}                                
\def\rad{\mathfrak{rad}}                          %
\def\sdp{{~\mathbin{{\triangleright}\!{<}}}}      
\def\siso{{\stackrel{\sim}{~\rightarrow~}}}       

\begin{abstract}
The work of Bernstein-Zelevinsky and Zelevinsky gives a good understanding of irreducible subquotients of 
a reducible principal series representation of $\GL_n(F)$, $F$ a $p$-adic field induced from an essentially discrete series representation of a parabolic subgroup (without specifying their multiplicities  which is done by  a Kazhdan-Lusztig type conjecture). In this paper we make a proposal 
of a similar kind for principal series representations of $\GL_n(\R)$. Our investigation on principal 
series representations naturally led us to consider the Steinberg representation for real groups, which has curiuosly 
not been paid much attention to in the subject (unlike the $p$-adic case). Our proposal for Steinberg is the 
simplest possible: for a real reductive group $G$, the Steinberg of $G(\R)$ 
is a discrete series representation if and only if $G(\R)$ has a discrete series, and makes up a full $L$-packet 
of representations of $G(\R)$ (of size $W_G/W_K$), so is typically not irreducible.
\end{abstract}

\tableofcontents

\section{Introduction}
Although irreducible representations of $\GL_n(\R)$ and $\GL_n(\C)$ are well understood, and so are their Langlands 
parameters, the author has not found a place which discusses Langlands parameters of irreducible subquotients of 
principal series representations on these groups. 
In fact, there is no explicit reference relating the Langlands parameters of 
the two components, one finite dimensional, and the other a discrete series representation of $\GL_2(\R)$
which appear inside a reducible principal series representation of $\GL_2(\R)$. The paper was conceived 
in the hope that this simple question may shed some light on possible Langlands parameters of reducible principal
series representations of $\GL_n(\R)$, or even more generally for $G(\R)$, for $G$ a reductive group over $\R$, and how such 
questions on real and $p$-adic groups may be related. It is well-known that reducibility of non-unitary 
principal series of $G(F)$, $F$ a $p$-adic field, is intimately connected with Langlands parameters involving the Weil-Deligne group. For example,
there is the well-known conjecture (usually attributed to Tom Haines) that for $F$ a $p$-adic field, 
the Langlands parameters of 
all the subquotients of a principal series representation induced from a cuspidal representation,
have the same restriction to $W_F$ when $W_F$ is embedded in $W'_F=W_F \times \SL_2(\C)$ in such a way that the
homomorphism of $W_F$ into $\SL_2(\C)$ lands inside the diagonal subgroup, and is the pair of characters $(\nu^{1/2},\nu^{-1/2})$ where $\nu: F^\times \rightarrow \C^\times$ is the normalized absolute value of $F^\times$, treated also as 
a character of $W_F$.

As an introduction to Langlands parameters for $\GL_n(\R)$, we recall that $\GL_n(\R)$ has a discrete series representation if and only if $n \leq 2$. Further, any tempered representation of $\GL_n(\R)$ is irreducibly induced from a (unitary) discrete series representation of a Levi subgroup. Thus the Langlands parameter $\sigma_\pi$ of any irreducible admissible representation $\pi$ of $\GL_n(\R)$ is of the form:
$$\sigma_\pi= \sum \sigma_i,$$ 
where $\sigma_i$ are irreducible representations of $W_\R$ of dimension $\leq 2$. Further, the map $\pi \rightarrow \sigma_\pi $ is a bijective correspondence between irreducible admissible representations of $\GL_n(\R)$ and (semi-simple) 
representations of $W_\R$ of dimension $n$.

We now begin with $\GL_2(\R)$. The following well-known proposition (in this form) is due to Jacquet-Langlands; in this, and in the rest of the paper, we denote by $\omega_\R$, the unique quadratic character of $\R^\times$.

\begin{proposition}For a pair of characters $\chi_1,\chi_2:\R^\times \rightarrow \C^\times$, let $Ps(\chi_1,\chi_2) = \chi_1 \times \chi_2$ be the corresponding
principal series representation of $\GL_2(\R)$. Then this principal series representation is reducible if and only if 
$\chi_1\cdot \chi_2^{-1}(t) =t^p\omega_\R(t)$ for some integer $p\not = 0$.

If the principal series representation is reducible,  it has two Jordan-H\"older factors, one finite dimensional (of dimension $|p|$), 
and the other infinite dimensional which
is a discrete series representation of $\GL_2(\R)$ by which we always mean up to a twist, or what's also called `an essentially discrete series representation'. 
\end{proposition}

\begin{cor} (a) For characters 
$\chi_1,\chi_2:\R^\times \rightarrow \C^\times$, with $\chi_1/\chi_2(t)=t^p$ for $t>0$ with $p\not = 0$, 
exactly one of the principal series  $\chi_1 \times \chi_2$ or $\chi_1 \times \chi_2 \omega_\R$ of $\GL_2(\R)$ is reducible,
and the other irreducible.

(b) For characters $\chi_1,\chi_2,\chi_3:\R^\times \rightarrow \C^\times$, 
among the three principal series  representations $\chi_1 \times \chi_2$, $\chi_1 \times \chi_3 $, $\chi_2 \times \chi_3$
of $\GL_2(\R)$, at least one of them is irreducible.
\end{cor}

\begin{remark} One can reformulate this proposition, or more generally the work of Speh on reducibility of principal series representations of
$\GL_n(\R)$ and $\GL_n(\C)$, see [Sp] as well as [Mo],
so that it applies uniformly to all $\GL_n(F)$, $F$ any local field. There is the conjecture 2.6 formulated in [GP] on 
when a representation $\pi$ is generic in terms of $L(s ,\pi, {\rm Ad}({\mathfrak g}))$ having no poles at $s=1$. Since the $L$-functions at archimedean places involve
Gamma functions, say $\Gamma(s/2)$, which has poles for all $s$ an even integer $s=2d \leq 0$, this is responsible for infinitely many reducibility
points for principal series representations of $\GL_2(\R)$, unlike the unique reducibilty point for $\GL_2(F)$, $F$ archimedean. Known reducibility points for  $\GL_n(\R)$ and $\GL_n(\C)$ 
were some of the initial examples which went into the formulation of Conjecture 2.6 of [GP].
\end{remark}

Recall that $W_{\R}= \C^\times \cdot \langle j \rangle$ 
with $j^2 =-1, jzj^{-1} = \bar{z}$ for $z \in \C^\times$.  Since $W_\R^{\rm ab} \cong \R^\times$ in which the map from $W_\R$ to $\R^\times$ 
when restricted to $\C^\times \subset W_\R$ is just the norm mapping from $\C^\times$ to $\R^\times$, characters of $W_\R \rightarrow \C^\times$ can therefore be identified to
characters of $\R^\times$.

Note that the characters of $\R^\times$ are of the form $\omega_\R^{\{0,1\}}(t) |t|^s$ where $s \in \C^\times$.   On the other hand, characters of $\C^\times$
can be written as,
$$z \rightarrow z^\mu \cdot \bar{z}^{\nu} = 
z^{\mu -\nu}\cdot  (z \bar{z})^\nu, \quad \mu,\nu \in \C, \mu -\nu \in \Z. $$

\begin{proposition}For a pair of characters $\chi_1,\chi_2:\R^\times \rightarrow \C^\times$ with 
$$\chi_1(t) = t^\mu,\quad  \chi_2(t) = t^{\nu}, \quad t > 0, \quad {\rm and} \quad \mu,\nu \in \C,$$
and with $\chi_1\cdot \chi_2^{-1}(t) =t^p\omega_\R(t)$ for  $p=\mu-\nu \not = 0$, an integer, let $Ps(\chi_1,\chi_2)$ 
be the corresponding
principal series representation of $\GL_2(\R)$.  Let  $F(\chi_1,\chi_2)$ be the finite dimensional subquotient of 
$Ps(\chi_1,\chi_2)$, and $Ds(\chi_1,\chi_2)$ the essentially discrete series component. Then the Langlands parameter of $F(\chi_1,\chi_2)$ 
is $\chi_1+ \chi_2$, and that of $Ds(\chi_1,\chi_2)$ is the induction to $W_\R$ of the following character of  $\C^\times \subset W_\R$:
$$z \longrightarrow z^\mu \bar{z}^\nu = 
z^{\mu -\nu}\cdot  (z \bar{z})^\nu.$$
\end{proposition}

\begin{proof}Observe that
  $\GL_2(\R) = \SL_2(\R)^{\pm 1} \times \R^{>0}$, where $\SL_2(\R)^{\pm 1}$ is the subgroup of $\GL_2(\R)$ 
consisting of matrices with determinant $\pm 1$. So any irreducible representation of $\GL_2(\R)$ restricted to $\SL_2(\R)^{\pm 1}$ 
remains irreducible, and the Langlands parameter of an  irreducible representation of $\GL_2(\R)$ can be read-off from that of  $\SL_2(\R)^{\pm 1}$.
It suffices then to note that the representation, say $\pi_n$, $n\geq 2$, of $\SL_2(\R)^{\pm 1}$ with lowest weight $n$ (extended to $\GL_2(\R)$ trivially across 
$\R^{>0}$) has Langlands parameter
which is the 2 dimensional irreducible representation of $W_\R$ given by
${\rm Ind}_{\C^\times}^{ W_\R} \chi_n$ 
where  $\chi_n$ is the character $ z = r e^{i \theta}\stackrel{\chi_n}\rightarrow e^{(n -1)i \theta}$. 
\end{proof}   
\begin{remark} For a discrete series representation $D$ of $\GL_2(\R)$, $D \otimes \omega_\R \cong D$.Therefore if a 
principal series representation 
$\chi_1 \times \chi_2$ of $\GL_2(\R)$ is reducible, and  
is $F+ D$ up to semi-simplification, where $F$ is finite dimensional and $D$ is a discrete series representation of $\GL_2(\R)$, then the principal series representation 
$\chi_1\omega_\R \times \chi_2\omega_\R$ of $\GL_2(\R)$ is also reducible, and is 
up to semi-simplification $\omega_\R F+ D$; i.e., a discrete series representation 
lies in two distinct principal series representations of $\GL_2(\R)$, whereas a finite dimensional 
representation lies in only one; this is in marked contrast with $p$-adics! (This difference between  reals and $p$-adics is at the source of not having a theory of `cuspidal supports' for $\GL_n(\R)$ although something slightly weaker given by 
{\it infinitesimal character} is there: thus if two principal series representations $\chi_1\times \chi_2 \times \cdots 
\times \chi_n$ and $\lambda_1 \times \lambda_2 \times \cdots \times \lambda_n$ of $\GL_n(\R)$ have a common irreducible 
subquotient, then $\lambda_i$ and $\chi_i$ are the same up to a permutation and possible multiplication by $\omega_\R$.)
\end{remark}

\begin{remark}Since $p$  is nonzero in the proposition, the character $z \longrightarrow z^\mu \bar{z}^\nu$ 
of $\C^\times$ induces an irreducible 2-dimensional
representation of $W_\R$.
Further, a well-known
property of Langlands parameter of representations of $\GL_n(\R)$ 
is that its determinant is the central character of the representation. This  holds in our case by the following
calculation on the
determinant of an induced representation 
$${\rm det} \left [{\rm Ind}_{\C^\times}^{W_\R} \chi 
\right ] = \chi|_{\R^\times} \cdot \omega_\R = t^{\mu + \nu}\omega_\R(t)= t^{\mu -\nu}(t^2)^\nu \omega_\R(t)=t^p(t^2)^\nu \omega_\R=
\chi_1(t)\cdot \chi_2(t),$$
where in the last equality we used $\chi_1\chi_2^{-1}(t) = t^p \omega_\R(t)$. 
\end{remark}

\section{Notion of Segments for $\GL_n(\R)$}  

We introduce the  language of segments 
for $\GL_n(\R)$
which 
will play a role similar to the language of segments for $\GL_n(F)$, $F$ a $p$-adic field, due to 
Zelevinsky,  for reducibility questions for representations of $\GL_n(\R)$ parabolically 
induced from essentially discrete series representations. 
We will introduce the  language of segments only for essentially discrete series representations of $\GL_n(\R)$ (so $n\leq 2$); one could extend the notion of segments to irreducible representations of $\GL_n(\R)$ for all $n$ using Langlands quotient theorem.
For this, we associate 
to the (essentially) 
discrete series representation $\pi_{\mu,\nu}$ of $\GL_2(\R)$ whose Langlands parameter restricted to $\C^\times$ contains the character
$z\rightarrow z^\mu \bar{z}^\nu$ with $\mu -\nu \in \N$, the segment $[\nu, \nu+1,\nu+2, \cdots,  \mu]$, written as $[\nu,\mu]$. 
We will denote the corresponding finite dimensional representation of $\GL_2(\R)$ by $F_{\mu,\nu}$. The discrete series 
representation $[\nu,\mu]$ is a subquotient of a principal series representation $\chi_1 \times \chi_2$ 
for an appropriate choice of characters $\chi_1,\chi_2:\R^\times \rightarrow \C^\times$ with $\chi_1(t)=t^\nu$ and
$\chi_2(t)=t^\mu$ for $t>0$.

If the discrete series appears in the reducible principal series 
$\nu^{s_j}(\nu^{p_j/2} \times \nu^{-p_j/2}\omega^{p_j+1}_\R)$ of $\GL_2(\R)$, 
the segment associated above is $s_j+I_j=[s_j-p_j/2, s_j-p_j/2+1 , \cdots, s_j+p_j/2 -1,   s_j+p_j/2]$ (the segment $I_j$ 
consists of integer translates of half-integers).

To the representation $\sigma_i\nu^{s_i}$ of $\R^\times$ where $\sigma_i$ is a character of order 1 or 2, 
we associate the segment $s_i+I_i$ with $I_i=0$ which is also to be denoted as $[s_i]$. This notation $[s_i]$ for a character
of $\R^\times$ determines the character only on $\R^+$, the rest left ambiguous.

It may be noted that we have segments of 
arbitrary length already for $\GL_2(\R)$, whereas for $F$ a $p$-adic field, 
to have segments of arbitrary length, we must 
have $\GL_m(F)$ for $m$ large enough; one may add that for Bernstein-Zelevinsky, a segment is `real' in that all the 
numbers in the segment correspond to representations, whereas a segment for $\GL_2(\R)$ has only end points which are meaningfully related to representations.

\section{The Steinberg representation for real groups} There is an  analogy between the representation 
$D_2: W_\R\rightarrow \SL_2(\C)$ associated to the lowest discrete series representation of $\PGL_2(\R)$, 
and the defining representation of $\SL_2(\C)$ of dimension 2 in the Deligne part of the 
Weil-Deligne  group $W'_F= W_F \times \SL_2(\C)$ for $F$ a non-archimedean local field. For this, consider 
the Steinberg representation $St_G$
of $G(\R)$ obtained on functions on $G(\R)/B$ for $B$ a minimal parabolic in $G$ modulo functions on $G(\R)/Q$ for $Q$ parabolics of $G$ strictly containing $B$.
It is known that $St_G$ is a tempered representation of $G(\R)$ of finite length which may not be irreducible unlike for $p$-adic fields as is the case already for $\SL_2(\R)$. 
If $F$ is a non-archimedean local field, then the Steinberg representation is an irreducible discrete series representation of 
$G(F)$, and is supposed to have the Langlands parameter $\imath_G:\SL_2(\C) \times W_F \rightarrow \widehat{G}(\C)$ which is trivial on $W_F$, and which on $\SL_2(\C)$ is associated by the Jacobson-Morozov theorem 
 to the regular unipotent element in the dual group $\widehat{G}(\C)$ 
(called the principal $\SL_2(\C)$) which is also a 
regular unipotent element in $\widehat{G}(\C)^{W_F}$ .

\begin{conjecture}
Let $G$ be a real reductive group, and $St_G$ the Steinberg representation of $G(\R)$. Then $St_G$ is a direct sum of irreducible tempered representations which makes up a full $L$-packet on $G(\R)$.

Let 
$D_2: W_\R\rightarrow \SL_2(\C)$ be the Langlands parameter 
associated to the lowest discrete series representation of $\PGL_2(\R)$. Let $\imath_G: \SL_2(\C)  \rightarrow {}^L G$ 
be the Jacobson-Morozov map
associated to a regular unipotent element in $\widehat{G}$.  
Then the Langlands parameter of $St_G$ is $\imath_G \circ D_2: 
W_\R \rightarrow {}^LG$. 
\end{conjecture}

\begin{example} It is well-known that if $V_2$ denotes the standard 2 dimensional representation of $\SL_2(\C)$, then the 
principal $\SL_2(\C)$ inside $\GL_{n+1}(\C)$ is given by ${\rm Sym}^n(V_2)$. Later we will have occasion to use the fact
that this principal $\SL_2(\C)$ sits inside the classical group $\Sp_{2n}(\C) \subset \GL_{2n}(\C)$ 
and inside $\SO_{2n+1}(\C) \subset \GL_{2n+1}(\C)$ in the two possible cases. 

Let $D_\ell$, $\ell \geq 2$,  be the 2 dimensional irreducible representation of $W_\R$ given by
${\rm Ind}_{\C^\times}^{ W_\R} \chi_\ell$ 
where  $\chi_\ell$ is the character of $\C^\times$ given by $ z = r e^{i \theta}\stackrel{\chi_\ell}\rightarrow e^{(\ell -1)i \theta}$. We will use the symbol $D_\ell$ to also denote 
the corresponding irreducible discrete series representation of $\GL_2(\R)$ which has trivial central 
character if $\ell$ is even, and has $\omega_\R$ as its central character for $\ell$ odd. 

It can be seen that, 
\begin{eqnarray*}
{\rm Sym}^{2n-1}(D_2) & =  & D_{2n}+D_{2n-2} + \cdots + D_2 \\
{\rm Sym}^{4n-2}(D_2) & =  & D_{4n-1}+D_{4n-3} + \cdots + \omega_\R \\
{\rm Sym}^{4n}(D_2) & =  & D_{4n+1}+D_{4n-1} + \cdots + 1,
\end{eqnarray*}
a sum of distinct irreducible representations of $W_\R$, 
which by our conjecture is the Langlands parameter of the Steinberg representation of $\GL_n(\R)$ for various integers $n$.
Our conjecture
further says that the Steinberg representation of $\GL_n(\R)$ is an irreducible representation of $\GL_n(\R)$. We will
prove the assertion on the Langlands parameter of the Steinberg representation of $\GL_n(\R)$ by an induction on $n$ (increasing in steps of 2) below. We first deal with $\GL_3(\R)$ which contains the basic idea.

We will show that the principal series representation $\nu \times 1 \times \nu^{-1}$ of $\GL_3(\R)$ contains the irreducible tempered 
representation $\omega_\R \times D_3$ as a subquotient 
which being an irreducible generic representation, is the Steinberg representation of $\GL_3(\R)$, and is further a submodule 
of the principal series representation $\nu \times 1 \times \nu^{-1}$ of $\GL_3(\R)$.

To see that the principal series representation $\nu \times 1 \times \nu^{-1}$ of $\GL_3(\R)$ 
contains the irreducible tempered 
representation $\omega_\R \times D_3$ of $\GL_3(\R)$, note that the principal series representation $\nu \times 1 \times \nu^{-1}$ 
contains the representation $\nu^{1/2}D_2 \times \nu^{-1}$ which is contained in 
the principal series representation $\nu\omega_{\R} \times \omega_\R \times \nu^{-1}$.  The principal series representations 
$\nu\omega_{\R} \times \omega_\R \times \nu^{-1}$ and  $\nu\omega_{\R} \times \nu^{-1} \times \omega_\R $ have the same Jordan-H\"older factors. The principal series representation $\nu\omega_{\R} \times \nu^{-1} \times \omega_\R $ contains 
$D_3 \times \omega_\R$ as a submodule 
which is irreducible being parabolically induced from a unitary representation. Since there is a unique Whittaker model in any principal series (irrespective of whether irreducible or not), it follows that     $D_3 \times \omega_\R$ is a subquotient of the principal series representation $\nu \times 1 \times \nu^{-1}$, proving the assertion on the Steinberg.

The essence of this analysis  is that  the principal series representation 
$\nu \times 1$ contains the  discrete series $\nu^{1/2}D_2$, but this discrete series is also contained in the
principal series $\omega_\R \nu \times \omega_\R$, so a certain part of the principal series representation 
$\nu \times 1 \times \nu^{-1}$ of $\GL_3(\R)$ 
is in common with the principal series representation 
$\omega_\R \nu \times  \omega_\R \times \nu^{-1}$ 
of $\GL_3(\R)$; more precisely, the two principal series representations 
$\nu \times 1 \times \nu^{-1}$ and $\omega_\R \nu \times  \omega_\R \times \nu^{-1}$ of $\GL_3(\R)$ have the same generic
representations which appears as a unique sub-quotient in either of the two principal series representations.

The analysis of the $\GL_3(\R)$  example can be used in an inductive way to prove 
 that the Langlands parameter of the Stenberg representation
of $\GL_n(\R)$ is indeed
\begin{eqnarray*}
{\rm Sym}^{2n-1}(D_2)  & =  & D_{2n}+D_{2n-2} + \cdots + D_2, \\
{\rm Sym}^{4n-2}(D_2) & =  & D_{4n-1}+D_{4n-3} + \cdots + \omega_\R, \\
{\rm Sym}^{4n}(D_2) & =  & D_{4n+1}+D_{4n-1} + \cdots + 1.
\end{eqnarray*}

For example, suppose we are in the first of the above cases, i.e. 
 for the group 
$\GL_{2n}(\R)$ in which case we are looking at 
the principal series representation
$$\nu^{-(2n-1)/2} \times \nu^{-(2n-3)/2} \times  \cdots \times  \nu^{(2n-3)/2} \times \nu^{(2n-1)/2},$$
and we need to prove that the irreducible (generic and tempered) representation with  parameter ${\rm Sym}^{2n-1}(D_2) $
is contained in this principal series representation.

By the inductive hypothesis, the principal series representation of $\GL_{2n-2}(\R)$
$$ \nu^{-(2n-3)/2} \times  \cdots \times  \nu^{(2n-3)/2} ,$$ 
contains the tempered representation  say $\pi_{2n-2}$ with parameter
$${\rm Sym}^{2n-3}(D_2)   =   D_{2n-2} + \cdots + D_2 .$$
Therefore the principal series representation of $\GL_{2n}(\R)$
$$\nu^{-(2n-1)/2} \times \nu^{-(2n-3)/2} \times  \cdots \times  \nu^{(2n-3)/2} \times \nu^{(2n-1)/2},$$ 
contains 
the representation 

$$\nu^{-(2n-1)/2} \times \pi_{2n-2} \times \nu^{(2n-1)/2}.$$
Since the principal series representation $\nu^{-(2n-1)/2} \times \nu^{(2n-1)/2}$ of $\GL_2(\R)$ contains the discrete series representation $D_{2n-2}$, it follows that
 the principal series representation of $\GL_{2n}(\R)$
$$\nu^{-(2n-1)/2} \times \nu^{-(2n-3)/2} \times  \cdots \times  \nu^{(2n-3)/2} \times \nu^{(2n-1)/2},$$ 
contains (as a subquotient) the irreducible generic representation of $\GL_{2n}(\R)$ with parameter
$${\rm Sym}^{2n-1}(D_2)   =   D_{2n}+D_{2n-2} + \cdots + D_2.$$

Since the Steinberg representation of $\GL_{2n}(\R)$ is the unique generic component of the principal series
$\nu^{-(2n-1)/2} \times \nu^{-(2n-3)/2} \times  \cdots \times  \nu^{(2n-3)/2} \times \nu^{(2n-1)/2},$  our argument in fact proves temperedness too of the Steinberg representation.

It may be noted that there is a way to come up with the parameter of the Steinberg representation of $\GL_n(\R)$ 
since it is built (as for any other representation) from parameters for $\GL_1(\R)$ and $\GL_2(\R)$, unlike the case of $p$-adics, where the parameter
of the Steinberg representation of $\GL_n(F)$, $n \geq 2$, is a `brand-new' parameter!

\end{example} 

\begin{remark}
As observed earlier, 
\begin{eqnarray*}
{\rm Sym}^{2n-1}(D_2)
  & =  & D_{2n}+D_{2n-2} + \cdots + D_2, \\
{\rm Sym}^{4n-2}(D_2) & =  & D_{4n-1}+D_{4n-3} + \cdots + \omega_\R, \\
{\rm Sym}^{4n}(D_2) & =  & D_{4n+1}+D_{4n-1} + \cdots + 1,
\end{eqnarray*}
a sum of distinct irreducible selfdual representations of $W_\R$ (of a fixed parity). Thus
${\rm Sym}^{i}(D_2)$  
gives rise to Langlands parameters not only for $\GL_{i+1}(\R)$ but also for $\Sp_{2n}(\R)$, and $\SO(p,q)(\R)$ as the case may be. The parameter 
${\rm Sym}^{i}(D_2)$ 
being a sum of distinct irreducible selfdual representations of $W_\R$ (of a fixed parity), the 
corresponding representations of the classical group are discrete series representations with Vogan 
$L$-packets of size $2^d$ where $d$ is either the number of irreducible summands in ${\rm Sym}^{i}(D_2)$ if $i$ is odd, or  
one less than the number of irreducible summands in ${\rm Sym}^{i}(D_2)$ if $i$ is even.
 
Our conjecture for $\Sp_{2n}, \SO_{2n+1}$ has the following two consequences:

\begin{enumerate}

\item The Steinberg representation for the groups $\Sp_{2n}(\R)$, 
and $\SO(p,q)(\R)$ ($p+q$ odd), is a 
discrete series representation.

\item The Steinberg representation for the groups $\Sp_{2n}(\R)$ is a sum of $2^n$ many irreducible discrete 
series representations of $\Sp_{2n}(\R)$ making up a full $L$-packet of size $|W_G/W_K|= 2^n$.

\item The Steinberg representation for the groups $\SO(p,q)(\R)$ 
($p+q$ odd) is a sum of $\frac{(p+q)!}{p! q!}$ 
many irreducible discrete 
series representations of $\SO(p,q)(\R)$ 
making up a full $L$-packet of size $|W_G/W_K|= \frac{(p+q)!}{p! q!}$.
\end{enumerate}

\end{remark}

\begin{proposition} Let $G$ be a semi-simple  group over $\R$. Let 
$D_2: W_\R\rightarrow \SL_2(\C)$ be the Langlands parameter 
associated to the lowest discrete series representation of $\PGL_2(\R)$. Let $\imath_G: \SL_2(\C) \rightarrow {}^LG$ 
be the Jacobson-Morozov map
associated to a regular unipotent element in $\widehat{G}$.
Then $\imath_G \circ D_2: W_\R \rightarrow {}^LG$ is a discrete parameter 
if and only if 
$G(\R)$ has a discrete series representation (a condition which depends only on the {\it inner-class} of $G$ over $\R$, equivalently only on the $L$-group of $G$).
\end{proposition}

\begin{proof}
We will prove that if $G(\R)$ has a discrete series representation then 
$\imath_G \circ D_2: W_\R \rightarrow {}^LG$ is a discrete parameter.

Recall that a parameter $\varphi: W_\R \rightarrow {} ^LG$ 
is said to be a discrete parameter if and only if 
the centralizer of $\varphi(W_\R)$ in $\widehat{G}$ is a finite group. It is easy to see that a parameter  $\varphi: W_\R \rightarrow {}^LG$ 
is a discrete parameter if and only if 
\begin{enumerate}
\item the centralizer of $\varphi(\C^\times)$ is a maximal torus $\widehat{T}$ in $\widehat{G}$, and 

\item for $j \in W_\R$ (the element outside $\C^\times$ in $W_\R$ with $j^2=-1$), the action of $\varphi(j)$ on $\widehat{T}$ 
is by $t\rightarrow t^{-1}$.
\end{enumerate}
Conversely, a group $G$ has a discrete series representation if there is an element in the normalizer 
of  the standard maximal torus $\wT$ in ${}^LG$  which acts by $t\rightarrow t^{-1}$ on $\wT$.

It is well-known and easy to see that for the principal $\SL_2(\C)$, $\imath_G: \SL_2(\C) \rightarrow {}^LG$, 
the centralizer in $\wG$ of the image of the diagonal torus of $\SL_2(\C)$ under $\imath_G$, and hence also 
the  centralizer in $\wG$ of the image of the subgroup of the diagonal torus of $\SL_2(\C)$ consisting of $\S^1=\{z \in \C^\times | |z|=1\}$ 
under $\imath_G$ equals $\wT$. 

We will re-write  ${}^LG=\wG \rtimes W_\R$, as ${}^LG=\wG \rtimes \langle j\rangle$, thus we give the symbol $j$ two (similar) 
roles, one as an element of $W_\R$, and the other as the corresponding element of ${}^LG$. 

To prove the proposition, it suffices to prove that the element $\imath_G\circ D_2(j)
\in {}^LG = \wG \rtimes W_\R$
which normalizes $\imath_G\circ D_2(\C^\times) \subset \wG$, hence normalizes its centralizer in $\wG$ which is $\wT$, acts on
$\wT$ by $t\rightarrow t^{-1}$. 

Since $G$ is given to have a discrete series representation, by what was recalled earlier, there is an element $g_0\cdot j 
\in {}^LG = \wG \rtimes W_\R$ 
which acts as $t\rightarrow t^{-1}$ on $\wT$. 
The element $\imath_G\circ D_2(j)$ 
acts as
$t\rightarrow t^{-1}$ on  $\imath_G\circ D_2(\S^1)$, 
which implies that the element $\imath_G\circ D_2(j) g_0^{-1}$ 
acts by identity  on $\imath_G\circ D_2(\S^1)$, and hence belongs to $\wT$, say $\imath_G\circ D_2(j)= g_0 t_0$ for 
some $t_0\in \wT$. This means that the two elements of ${}^LG = \wG \rtimes W_\R$ given by 
$g_0\cdot j$ and $\imath_G\circ D_2(j) \cdot j$ 
differ by an element ($j^{-1}t_0j$) in $\wT$ since:
$$ \imath_G\circ D_2(j) \cdot j = g_0t_0j= g_0 j j^{-1}t_0j.$$
The element $g_0\cdot j$ is given to act
as $t\rightarrow t^{-1}$ on $\wT$ so is the case for $\imath_G\circ D_2(j) \cdot j$ too, 
proving the proposition. 
 \end{proof}

\section {Paraphrasing the work of Speh} 
We now recall the thesis work of B. Speh [Sp] on reducibility of principal series representations of $\GL_n(\R)$, and then 
rephrase it using a language of segments similar to that due to Bernstein-Zelevinsky for $\GL_n(F)$, for $F$, $p$-adics.
The work of Speh is not published; we will follow the exposition of Moeglin [Mo] on Speh's work closely.

For $j=1,\cdots, t$, let $n_j=1$ or 2  be integers, with $\sum n_j = n$, and $s_j \in \C$. If $n_j=1$, fix also a character $\sigma_j$ of order 1 or 2. If $n_j =2$, fix an integer $p_j \geq 1$, and let $\sigma_j$ be the discrete series representation 
contained in the reducible principal series representation $\nu^{p_j/2} \times \nu^{-p_j/2}\omega^{p_j+1}_\R$.

Write $\underline{\chi}$ for the collection of triples,
$$\underline{\chi} = \{(n_j,s_j,\sigma_j), j = 1,\cdots ,t\}.$$
 
Denote by $I(\underline{\chi})$ 
the representation of $\GL_n(\R)$ induced from a parabolic subgroup of $\GL_n(\R)$ with Levi subgroup which is 
$\GL_{n_1}(\R) \times \cdots \times \GL_{n_t}(\R)$ of the representation:
$$\bigotimes \sigma_j \otimes \nu^{s_j}.$$

\begin{theorem} If $\underline{\chi}$ is as above, $I(\underline{\chi})$ is irreducible if and only if for all 
$i,j \in [1,t]$ with $n_i \geq n_j$, either $s_i-s_j \not \in \frac{1}{2}\Z$, 
or the appropriate one of the following conditions 
is satisfied (thus we are also assuming that $s_i-s_j \in \frac{1}{2}\Z$):
\begin{enumerate}
\item If $n_i=n_j=1$, $|s_i-s_j|$ is not an even 
(resp. odd integer)  nonzero integer if and only if $\sigma_i \not = \sigma_j$ (resp $\sigma_i = \sigma_j$).

\item If $n_i=2$ and $n_j=1$ (so that $p_i$ is defined), $-p_i/2+ |s_i-s_j| \not \in \{1,2,3,\cdots \}$. 

\item If $n_i=n_j=2$ (so that $p_i, p_j$ are defined), $-|p_i-p_j|/2+ |s_i-s_j| \not \in \{1,2,3,\cdots \}$. 
\end{enumerate}
\end{theorem}

It is hoped that the following paraphrase of the previous theorem is simpler to use.

\begin{theorem} 
Consider a principal series representation  $Ps(\underline{\pi}) = \pi_1 \times  \cdots  \times  \pi_t$  of $\GL_n(\R)$
where $\pi_i$ are essentially discrete series representations on $\GL_{n_i}(\R)$ (thus $n_i \leq 2$). 
Assume that the segments associated to the representations $\pi_i$ are $[a_i,b_i]$. 
Then, $Ps(\underline{\pi})$ is irreducible if and only if for all pairs 
$i,j \in [1,t]$ with $i \not = j$ (and assuming that we are not in the case that $n_i=n_j=1$, a case corresponding to 
reducibility of principal series handled by Proposition 1), 
 \begin{enumerate}
\item  $a_i-a_j \not \in \Z$. Or,

\item $a_i-a_j \in \Z$, and one  of the two segements:  
$[a_i,b_i], [a_j,b_j]$  is contained in the other.
\end{enumerate}
\end{theorem}

\begin{remark} The paraphrase of Speh's theorem above in terms of segements says that for essentially 
discrete series representations $\pi_1$ and $\pi_2$ of $\GL_{n_1}(\R)$ and $\GL_{n_2}(\R)$, for the representation  $\pi_1 \times \pi_2$ 
of $\GL_{n_1+n_2}(\R)$ to be irreducible, either the segments are not integrally related, or 
one of the segments is contained in the other. 
Thus the essential difference with the BZ criterion for $p$-adic fields  is that 
for $p$-adic, for $\pi_1 \times \pi_2$ is irreducible if  either  (a) one of the segments is  contained in the other, or (b) the segments are disjoint and not linked, whereas for $F=\R$, it is only (a) which is responsible for irreducibility. 
\end{remark}

\section{A partial order on the set of irreducible representations}
 Let $G$ be a reductive algebraic group over a local field $F$. Introduce a relation $\geq$  on the set of irreducible admissible representations 
of $G(F)$, 
by saying
that $\pi_1 \geq  \pi_2$ if $\pi_1$ appears as the JH factor of the standard
module in which $\pi_2$
 is a Langlands quotient. In the case of $\GL_n(F)$, since there is a bijective correspondence between 
irreducible admissible representations of $\GL_n(F)$, and $n$-dimensional (admissible) representations of $W'_F$, there is a 
partial order on the set  of $n$-dimensional Langlands parameters too. Here is a most natural question which we do not know 
how to answer.

\vspace{4mm}

\begin{conjecture}
The relation introduced above is an equivalence relation, i.e., if $\pi_1 \geq \pi_2$ and $\pi_2 \geq \pi_3$, then $\pi_1 \geq \pi_3$.
\end{conjecture}

\section{Principal series representations for $\GL_3(\R)$, $\GL_4(\R)$}
\begin{lemma} Let $a,b,c$ be three complex numbers lying on a $\Z$-line, i.e., $(a-b) \in \Z$ as well as $(b-c) \in \Z$. Assume 
that $(a-b) \in \N=\Z^{>0}$ and $(b-c) \in \N=\Z^{>0}$. Then a principal series 
representation $[a,b] \times [c]$ of $\GL_3(\R)$ 
induced from an essentially discrete series representation of $(2,1)$ parabolic
is of length 2, with the generic subquotient the irreducible representation $[a,c] \times [b]$ 
which is an irreducible representation; here $[b]$ denotes the unique character of $\R^\times$ which is $\nu^b$ on $\R^+$, and is determined on all of 
$\R^\times$ by the equality of the central characters of the representations $[a,b] \times [c]$ and
$[a,c] \times [b]$ of $\GL_3(\R)$.  
\end{lemma}
\begin{proof} We will use standard methods of Whittaker model, well-known for $p$-adic fields, and surely also known 
for  archimedean fields.

The principal series representation $[a,b] \times [c]$ 
of $\GL_3(\R)$ is a sub-quotient of the principal series representation
$a \times b \times c$, 
and the generic component of $a \times b \times c$ is contained inside $[a,b] \times c$. However,
the generic component of $a \times b \times c$ is also contained inside $[a,c]  \times b$ which by our assumptions on $a,b,c$ 
and the theorem of Speh
is an irreducible representation of $\GL_3(\R)$.
Therefore the principal series reprepresentation $[a,b] \times [c]$ must contain $[a,c] \times [b]$ as a sub-quotient.
\end{proof}

The following conjecture is inspired by conversations  with C. Moeglin and D. Renard. Given that it is a rather explicit 
decomposition of a principal series representation of $\GL_4(\R)$, it is surprising that it is not proved by anyone, or may be it is actually proven by Speh in her thesis as she does prove many results of this nature, but it is difficult to interpret her results in the form that we want here.

\begin{conjecture} Let $\pi_1 = [a,b]$ be the essentially discrete series representation of $\GL_2(\R)$ corresponding to the representation $\sigma_{a,b}$ of $W_\R$ 
whose restriction to $\C^\times$ contains the character $z\rightarrow z^a\bar{z}^b$ for complex numbers $a,b$ with $a-b \in \N$. Similarly 
let $\pi_2 = [c,d]$ be an essentially discrete series representation of $\GL_2(\R)$. 

\begin{enumerate}
\item Assume that $(a-c) \in \N$ with $(a-c) > (a-b)$, thus with
$$a> b>c>d.
$$
(An inequality between complex numbers with the same imaginary parts, whose real parts have this inequality.) 

Then the principal series representation $\pi_1 \times \pi_2$ of $\GL_4(\R)$ has length 5 whose Jordan-H\"older factors 
have  the following Langlands parameters, each
appearing with multiplicity 1:
\begin{enumerate}
 \item $\sigma_{a,b} + \sigma_{c,d}$.
\item $\sigma_{a,c} + \sigma_{b,d}$.
\item $\sigma_{a,d}+ \sigma_{b,c}$.
\item $\sigma_{a,d} + \nu^{b}\omega_\R + \nu^c$ 
if $b-c$ is an even integer, else  $\sigma_{a,d} + \nu^{b} + \nu^c$.

\item $\sigma_{a,d} + \nu^{b} + \nu^c\omega_\R$ if $b-c$ is an even integer, else  $\sigma_{a,d} + \nu^{b} \omega_\R + 
\nu^c \omega_\R$. 
\end{enumerate}

\item If the segments $[a,b]$ and $[c,d]$ are linked but not juxtaposed, so assuming without loss of generality that 
$(a-c) \in \N$ with 
$(a-b)\geq (a-c)$, then the principal series representation 
$[a,b] \times [c,d]$ of $\GL_4(\R)$ 
has length 2. Assuming that $(a-c) \in \N$, the Langlands parameter
of the generic component of the principal series representation $[a,b] \times [c,d]$ of $\GL_4(\R)$ is 
$$\sigma_{a,d}+ \sigma_{b,c}.$$
 \end{enumerate}
\end{conjecture}
\begin{conjecture}
Consider a standard module  $Ps(\underline{\pi}) = \pi_1 \times  \cdots  \times  \pi_t$  of $\GL_n(\R)$
where $\pi_i$ are essentially discrete series representations of $\GL_{n_i}(\R)$ (thus $n_i \leq 2$) with Langlands 
parameters $\sigma_i$. 
Then an irreducible admissible representation $\pi$ of $\GL_n(\R)$ appears in $Ps(\underline{\pi})$ as a 
Jordan-H\"older factor if and only if its Langlands parameter 
$\sigma_\pi$ is obtained from  $\tau_0= \sigma_1+ \sigma_2 + \cdots + \sigma_n$ by a sequence of operations starting with $\tau_0$ in which one goes from $\tau_i$ to $\tau_{i+1}$ by replacing a summand of $\tau_i$ of the form $\sigma$ by $\sigma'$ for 
$\dim \sigma = \dim \sigma' \leq 4$, and $\sigma' \geq \sigma$.
\end{conjecture}

(The conjecture amounts  to say that any JH factor of a standard module for $\GL(n,\R)$ can be seen by analysing the JH components for $\GL(m,\R)$ for $m \leq 4$.)

\section{An Example}

According to Bernstein-Zelevinsky and Zelevinsky, the simplest principal series representations to analyze for $\GL_n(F)$, $F$ a 
$p$-adic field, is the principal series representation $Ps(\chi)$ 
induced from a character $\chi: (F^\times)^n \rightarrow \C^\times$ with the property
$\chi^w \not = \chi$ for $w \not = 1$. These principal series representations 
have Jordan-H\"older factors of multiplicity 1 
and can be explicitly described in terms of 
Zelevinsky classification, cf. [Ze]. Also, the 
Langlands parameters of all subquotients 
can be explicitly described, see for instance [Ku]. The corresponding questions for $\GL_n(\R),\GL_n(\C)$ 
have to my knowledge not been
attempted, although one does know --- by the thesis work of B. Speh [Sp] --- 
exactly when a principal series representation of $\GL_n(\R),\GL_n(\C)$ is reducible. For an exposition of the work of Speh, see [Mo]. For $p$-adic fields,
Langlands parameters of all subquotients of a principal series representation $Ps(\chi)$ are intimately linked with the notion of {\it Weil-Deligne} group
$W'_F = W_F \times \SL_2(\C)$. This group is not available to us for real groups. The Langlands parameters of the 
subquotients of the simplest principal series 
representation realized on functions of $G(\R)/B(\R)$ do not seem to be known. 

\vspace{2mm}

The space of functions on $\GL_{n}(F)/B(F)$ 
is the principal series representation which is 
$\nu^{-(n-1)/2} \times \nu^{-(n-3)/2} \times  \cdots \times  \nu^{(n-3)/2} \times \nu^{(n-1)/2}$. 
If $F$ is  $p$-adic, this principal series representation has $2^{n-1}$ many Jordan-H\"older factors,  each appearing with multiplicity 1. The
Langlands parameters of these Jordan-H\"older factors are obtained by 
writing the interval $[{-(n-1)/2},{-(n-3)/2}, \cdots,  {(n-3)/2}, {(n-1)/2}]$ as disjoint union 
of $r$ many non-empty intervals for all $0\leq r \leq n$. (These Jordan-H\"older factors are parametrized by all parabolics $P$ 
containing a given Borel subgroup $B$, and are given by a 
Steinberg-like construction 
by considering functions on $G/P$ modulo functions on $G/Q$ for all parabolics $Q$ strictly containing $P$. 
Irreducibility of such representations is a theorem of Casselman, cf. [Ca].) 

For each such interval $[i,i+1,\cdots, i+r_i]$, we have the representation
of the Weil-Deligne group $W_F \times \SL_2(\C)$ which is given by $\nu^{r_i/2+ i} \boxtimes {\rm Sym}^{r_i}(\C^2)$, 
and 
the Langlands parameter of the corresponding 
representation of $\GL_{n}(F)$ to be: 
$$\sum \nu^{r_i/2+ i} \otimes {\rm Sym}^{r_i}(\C^2).$$

What about $\GL_{n}(\R)$? Note that unlike the $p$-adic case, the number of 
Jordan-H\"older factors of this principal series representation of $\GL_n(\R)$ is much more complicated
for $n \geq 4$, for example there are 11 Jordan-H\"older factors of  this principal series representation of $\GL_4(\R)$.
However, it is not clear if all is lost in going from $p$-adics to reals, for example I could not find out if 
the principal series 
 representations $Ps(\chi)$ 
induced from a character $\chi: (F^\times)^n \rightarrow \C^\times$ with the property
$\chi^w \not = \chi$ for $w \not = 1$ have
Jordan-H\"older factors of multiplicity 1 (which for $p$-adic fields 
is a consequence of calculations with the Jacquet-module which is not as developed for real groups).

\vspace{4mm}

\noindent{\bf Concluding remark:} The first version of this paper was submitted to the arXiv last summer. It was much off the mark especially in my conclusions on subquotients of a standard module for $\GL_4(\R)$, and was naturally received with much sarcasm by the experts who opined that perhaps there is no simple answer to possible Jordan-H\"older factors of 
principal series representations being related to Kazhdan-Lusztig-Vogan polynomials. That may be the case, still
 I am happy to raise some very obvious questions, and also to 
propose some obvious answers for them, even if they are still off the mark.

\vspace{.5cm}
\noindent
Dipendra Prasad

\noindent
Tata Insititute of Fundamental Research,
Colaba,
Mumbai-400005, INDIA.

\noindent
e-mail: {\tt dprasad@math.tifr.res.in}

\end{document}